\documentclass[a4paper, english, 11pt, twoside]{amsart}

\usepackage{stmaryrd,amscd,amsmath,amsbsy,hyperref,ulem}
\usepackage{amssymb}
\usepackage{amsfonts}
\usepackage{xcolor}
\usepackage{bbold}
\usepackage[notcite,notref,final]{showkeys}

\newtheorem{theorem}{Theorem}[section]
\newtheorem{lemma}[theorem]{Lemma}
\newtheorem{proposition}[theorem]{Proposition}

\theoremstyle{definition}
\newtheorem{definition}[theorem]{Definition}

\newtheorem{question}[theorem]{Question}

\newtheorem{remark}[theorem]{Remark}

\newtheorem*{thintro}{Theorem}
\newtheorem*{propintro}{Proposition}
\newtheorem*{questintro}{Question}

\newcommand{\NN}{{\mathbb{N}}}
\newcommand{\ZZ}{{\mathbb{Z}}}

\newcommand{\Oint}{\mathcal{O}}

\newcommand{\FF}{{\mathbb{F}}}
\newcommand{\Weyl}{{\bold A}_n}
\newcommand{\Spec}{\textrm{Spec}\,}
\newcommand{\Ima}{\textrm{Im}\,}

\def\pf{\begin{proof}}
\def\epf{\end{proof}}

\begin{document}

\title[Semisimple Hopf actions on Weyl algebras]{Semisimple Hopf actions on Weyl algebras}

\author{Juan Cuadra}
\address{Department of Mathematics, University of Almeria,
Spain}
\email{jcdiaz@ual.es}

\author{Pavel Etingof}
\address{Department of Mathematics, Massachusetts Institute of Technology,
Cambridge, MA 02139, USA}
\email{etingof@math.mit.edu}

\author{Chelsea Walton}
\address{Department of Mathematics, Massachusetts Institute of Technology,
Cambridge, MA 02139, USA}
\email{notlaw@math.mit.edu}

\subjclass[2010]{12E15, 13A35, 16T05, 16W70}
\keywords{division algebra, Hopf algebra action, reduction modulo p, Weyl algebra}

\maketitle

\begin{abstract}
We study actions of semisimple Hopf algebras $H$ on Weyl algebras $A$ over an algebraically closed field of characteristic zero. We show that the action of $H$ on $A$ must factor through a group action; in other words, if $H$ acts inner faithfully on $A$, then $H$ is cocommutative. The techniques used include reduction modulo a prime number and the study of semisimple cosemisimple Hopf actions on division algebras.
\end{abstract}

\section{Introduction}

Let $k$ be an algebraically closed field of characteristic zero and $H$ a semisimple Hopf algebra over $k$. In \cite[Theorem 1.3]{EW1}, two of the authors showed that any action of $H$ on a commutative domain over $k$ factors through a group action. The goal of this paper is to extend this result to Weyl algebras. Our main result states:

\begin{thintro}[Theorem~\ref{main1}] \label{main1intro} {\it Any semisimple Hopf action on the Weyl algebra $\Weyl(k)$ factors through a group action.}
\end{thintro}

An equivalent formulation would be the following: if $H$ acts {\it inner faithfully} on $\Weyl(k)$, then $H$ is cocommutative. By definition, inner faithfulness means that the action of $H$ does not factor through a quotient Hopf algebra of smaller dimension.

Note that when the action of $H$ preserves the standard filtration of $\Weyl(k)$, Theorem \ref{main1} can be deduced from \cite[Proposition 5.4]{EW1}, since the associated graded algebra ${\rm gr}(\Weyl(k))$ is a commutative domain. Our main achievement in this paper is to eliminate this assumption.

We also obtain the result above for $H$ finite dimensional, not necessarily semisimple, provided that the action of $H$ gives rise to a Hopf-Galois extension, see Theorem \ref{main1a}.

The proof of Theorem \ref{main1} relies on reduction modulo a prime number, which allows us to reduce to the case where the algebra satisfies a polynomial identity (or is PI, for short). In this case, its quotient field is a division algebra with an action of $H$ (Lemma \ref{lem:quotdiv}). We then use the following result, interesting by itself:

\begin{propintro}[Proposition \ref{noaction}(ii)] \label{divisionintro}
{\it Let $H$ be a semisimple cosemisimple Hopf algebra of dimension $d$ over an algebraically closed field $F$ (of any characteristic). Let $D$ be a division algebra over $F$ of degree $m$. If $d!$ is coprime to $m$, then any action of $H$ on $D$ factors through a group action.}
\end{propintro}

Using these methods, we will establish more general results on semisimple and nonsemisimple
Hopf actions on quantized algebras in future work. In particular, these methods will apply to module algebras $B$ so that:

\begin{enumerate}
\item[($\dag$)] $B_p$, the reduction of $B$ modulo a prime number $p$, is PI and the PI-degree of $B_p$ is a power of $p$, for $p \gg 0$.
\end{enumerate}

\noindent Such algebras include universal enveloping algebras of finite dimensional Lie algebras and algebras of differential operators of smooth irreducible affine varieties. This prompts the following question, which is of independent interest in Ring Theory.

\begin{questintro}
{\it Let $B$ be a $\mathbb{Z}_+$-filtered algebra over $k$ with gr($B$) a finitely generated commutative domain. Does ($\dag$) hold for any large prime $p$?}
\end{questintro}

The paper is organized as follows. We recall preliminary results on reducing Hopf actions to positive characteristic in Section \ref{sec:modp}. In Section \ref{sec:division}, we prove  results on semisimple cosemisimple Hopf actions on division algebras; in particular, we establish
Proposition \ref{noaction}. We prove Theorems \ref{main1} and \ref{main1a} in Section \ref{sec:Weyl}. \medskip

{\bf Notation.} Throughout this paper $k$ is an algebraically closed field of characteristic zero, and $H$ is a Hopf algebra over $k$ of finite dimension $d$. For $n \in \NN$ and a commutative ring $R$ recall that the $n$-th Weyl algebra $\Weyl(R)$ is the $R$-algebra generated by $x_i,y_i \ (i=1,\dots,n)$ subject to the relations $[x_i,x_j]=[y_i,y_j]=0$ and $[y_i,x_j]=\delta_{ij}$.

\section{Reducing Hopf actions modulo a prime} \label{sec:modp}

In this section, we will show that given an action of $H$ on $A:={\bold A}_n(k)$ we can reduce it to positive characteristic (Proposition \ref{prop:modp}).
This is done in a standard way, as one does for any kind of ``finite" linear algebraic structure. To explain this we first need the notion of a Hopf order over a subring $R$ of $k$. See \cite{La} or \cite{CM} for details.

Given a finite dimensional $k$-vector space $V$, an {\it $R$-order} of $V$ is a finitely generated and projective $R$-submodule $V_R$ of $V$ such that the natural map $V_R \otimes_R \, k \rightarrow V$ is an isomorphism. We can now define an order of any algebraic structure existing on $V$ as an order of $V$ closed under the structure maps. The linear isomorphism $V_R \otimes_R k \rightarrow V$ will become an isomorphism for that structure. In particular, we have:

\begin{definition} \label{R-order}
A {\it Hopf $R$-order} of a Hopf algebra $H$ is an order $H_R$ of $H$ such that $1_H \in H_R$, $~H_RH_R \subseteq H_R$, $~\Delta(H_R)\subseteq H_R \otimes_{R} H_R$, $~\varepsilon(H_R) \subseteq R$ and $S(H_R)\subseteq H_R$.
\end{definition}

\begin{lemma} \label{lem:R-order}
Let $H$ be a finite dimensional Hopf algebra over $k$. Assume that $A$ is endowed with an action $\cdot : H \otimes_k A\to A$. Then, there is a finitely generated subring $R$ of $k$ and a Hopf $R$-order $H_R$ of $H$ such that $\cdot$ restricts to an action $\cdot_R:H_R \, \otimes_R A_R \to A_R$, where $A_R:=\Weyl(R)$.
\end{lemma}

\pf
Let $\{h_i\}_{i=1}^d$ be a basis of $H$. We have
$$
h_i \, \cdot \, x_j =P_{ij}(x_1,\dots,x_n,y_1,\dots,y_n) \hspace{7pt} \textrm{and} \hspace{7pt} h_i \, \cdot \, y_j=Q_{ij}(x_1,\dots,x_n,y_1,\dots,y_n),
$$
where $P_{ij},Q_{ij}$ are certain noncommutative polynomials over $k$. Let $R$ be any finitely generated subring of $k$ containing both the structure constants of $H$ in $\{h_i\}$ and the coefficients of $P_{ij}$ and $Q_{ij}$. Let $H_R=\bigoplus_{i=1}^d Rh_i$.
Then, $R$ and $H_R$ are as required.
\epf

To pass to positive characteristic, we need the following lemma from Commutative Algebra, which is standard but we provide a proof for the reader's convenience.

\begin{lemma} \label{lemm2}
Let $R$ be a finitely generated subring of $k$. Then: \vspace{-2pt}
\begin{enumerate}
\item[(i)] For a sufficiently large prime number $p$, the set $\Spec R(\overline{\FF}_p)$ is non\-emp\-ty. That is, there exists a homomorphism $\psi: R \to \overline{\FF}_p$.
\item[(ii)] For any prime number $\ell$, the homomorphism
$$\Psi_\ell:=\prod_{p\ge \ell}\ \prod_{\psi\in \Spec\! R(\overline{\FF}_p)}\psi: R \longrightarrow \prod_{p\ge \ell}\ \prod_{\psi\in \Spec\! R(\overline{\FF}_p)}\overline{\FF}_p$$
is injective.
\end{enumerate}
\end{lemma}

\pf
(i) This is a special case of Chevalley's constructibility theorem for schemes, applied to the natural morphism $\pi: \Spec R \to \Spec \ZZ$, see, e.g., \cite[1.8.4]{EGA}.
To prove it, note that by the Nullstellensatz, there is a homomorphism $\phi: R \to \overline{\Bbb Q}$. Since $R$ is finitely generated, $\Ima\phi$ is contained in some number field $L\subset \overline{\Bbb Q}$; and moreover, in $\Oint_{L}[1/r]\subset L$, where $\Oint_{L}$ is the ring of integers of $L$, and $r \in \NN$. Set $S=\Oint_{L}[1/r]$. For any prime $p$ not dividing $r$, we have $p\Oint_{L}=(pS) \cap \Oint_{L}$, and therefore $S/pS=\Oint_{L}/p\Oint_{L}$, which is a finite dimensional commutative $\Bbb F_p$-algebra. \par \smallskip

(ii) Let $x\in R$ be such that $\Psi_\ell(x)=0$. Consider the element $\phi(x)$ for a homomorphism $\phi: R\to \Oint_{L}[1/r]$ as above. We have $\phi(r^mx) \in \Oint_{L}$ for some $m\in \NN$. For all sufficiently large $p$, the algebra $\Oint_{L}/p\Oint_{L}$ is a direct sum of finite fields of characteristic $p$. Since $\Psi_\ell(x)=0$, the projection of $\phi(r^mx)$ to $S/pS=\Oint_{L}/p\Oint_{L}$ is zero. This implies that $\phi(x)=0$ because $\Oint_{L}$ is a finitely generated abelian group. As this is satisfied for all choices of $\phi$, the Nullstellensatz implies that $x=0$, as claimed.
\epf

Now using Lemma \ref{lem:R-order} and Lemma \ref{lemm2}, we can define the {\it reduction of $H$ modulo $p$} by the formula $$H_{\psi,p}:=H_R \otimes_R \overline{\FF}_p,$$ where the action of $R$ on $\overline{\FF}_p$ is via a homomorphism $\psi\in \Spec R(\overline{\FF}_p)$.
When no confusion is possible, we will simply write $H_p$ instead of $H_{\psi,p}$.

\begin{proposition} \label{prop:modp}
For a sufficiently large prime $p$: \vspace{-2pt}
\begin{enumerate}
\item[(i)] The algebra $A_p:=\Weyl(\overline{\FF}_p)$ admits an action of the Hopf algebra $H_p$.
\item[(ii)] The action of $H_p$ on $A_p$ is inner faithful when the action of $H$ on $A$ is inner faithful.
\end{enumerate}
\end{proposition}

\begin{proof}
(i) The action $\cdot_p: H_p \otimes A_p \to A_p$ is obtained by tensoring the action $\cdot_R:H_R \otimes A_R\to A_R$ with $\overline{\FF}_p$ over $R$ using $\psi: R \to \overline{\FF}_p$. \par \smallskip

(ii) Let us first show that $H$ acts faithfully on $A^{\otimes s}$ for some $s$. Let $K_{s} \subset H$ be the kernel of the action of $H$ on $A^{\otimes s}$. Observe that $K_{s} \supset K_{s+1}$ because $A^{\otimes s}=A^{\otimes s}\otimes 1\subset A^{\otimes s+1}$. Let $K=\bigcap_{s} K_{s}$. There is an integer $s_0$ such that $K=K_s$ for all $s \ge s_0$. Given $h \in K$, consider the action of $\Delta(h)$ on $A^{\otimes s}\otimes A^{\otimes t}$ for $s,t\geq s_0$. Since $A^{\otimes s}\otimes A^{\otimes t}$ is a faithful module
over $H/K\otimes H/K$, we find that $\Delta(h)\in K\otimes H+H\otimes K$. Thus, $K$ is a bialgebra
ideal of $H$, hence a Hopf ideal by \cite[Proposition 7.6.1]{Rad}. Since $H$ acts on $A$ inner faithfully, this implies that $K=0$,
as claimed.

Now we reduce the faithful action of $H$ on $A^{\otimes s}$ above modulo $p$ as follows. Using that $H$ is finite dimensional, there exist $v_1,\ldots,v_{q} \in A_R^{\otimes s}$ and $R$-linear maps $f_{ij}: A_R^{\otimes s} \to R$, $i=1,\ldots,d$, $j=1,\ldots,q$, such that the matrix with entries  $b_{il}:=\sum_j f_{lj}(h_i\cdot v_j)$ has a nonzero determinant $b:=\det(b_{il})\in R$. Then $b$ is invertible modulo $p$ when $p$ is sufficiently large (namely, does not divide the norm of $b$). So,
the matrix $(b_{il})$ is nondegenerate modulo $p$, which implies that $H_p$ acts faithfully on $A_p^{\otimes s}$.

Finally, note that any Hopf ideal of $H_p$ annihilating $A_p$ would also annihilate $A_p^{\otimes s}$. So, $H_p$ acts inner faithfully on $A_p$.
\end{proof}

We will also need the following lemma:

\begin{lemma} \label{lem:Hmodp}
If $H$ is semisimple (hence cosemisimple), then for a sufficiently large $p$, the Hopf algebra $H_p$ over $\overline{\FF}_p$ is semisimple and cosemisimple.
\end{lemma}

\pf
Since $H$ is a semisimple algebra over an algebraically closed field, it is separable. So, the multiplication map $\mu: H\otimes H\to H$ admits a splitting map of $H$-bimodules $\vartheta: H\to H\otimes H$. Reducing $\vartheta$ modulo $p$ for $p$ sufficiently large, we see that the same fact holds for $H_p$. Therefore, $H_p$ is separable, and hence semisimple. The same argument may be applied to $H^*$ to prove that $H_p$ is cosemisimple.
\epf

\section{Semisimple cosemisimple Hopf actions on division algebras} \label{sec:division}

In this section, we establish our result on actions of semisimple cosemisimple Hopf algebras on division algebras (Proposition \ref{noaction}(ii) below). Here, we work over an algebraically closed field $F$ of arbitrary characteristic.

\begin{lemma} \label{lem:quotdiv}
Let $H$ be a Hopf algebra over $F$. If a PI domain $B$ over $F$ admits an inner faithful action of $H$, then so does its quotient division algebra ${\mathcal Q}_B$.
\end{lemma}

\pf Since $B$ is a PI domain, after localization we obtain a classical quotient ring ${\mathcal Q}_B$, which is a division algebra; see \cite[Corollary~7.5.2]{Co} and \cite[Theorem~6.8]{GW}. Now apply \cite[Theorem 2.2]{SV} to obtain the result.
\epf

Let $D$ be a division algebra over $F$ that carries an action of $H$. Denote by $D^H$ the subalgebra of $H$-invariants in $D$. Moreover, for a division subalgebra $C$ of $D$, let $[D:C]_l$ and $[D:C]_r$ denote the left and right dimensions of $D$ over $C$, respectively.

\begin{lemma}\label{dens} \cite[Corollary~2.3]{BCF} One has
$$[D:D^H]_l \leq d \quad \textstyle{and} \quad [D:D^H]_r \leq d.$$

\vspace{-.3in}
\qed
\vspace{.1in}
\end{lemma}

\begin{proposition}\label{noaction}
Let $H$ be a Hopf algebra of dimension $d$ over an algebraically closed field $F$. Let $D$ be a division algebra over $F$ of degree $m$. If ${\rm gcd}(d!,m)=1$, then:
\begin{enumerate}
\item[(i)] The center $Z$ of $D$ is $H$-stable, and $D=ZD^H$. \smallskip
\item[(ii)] If $H$ is semisimple and cosemisimple, any action of $H$ on $D$ factors through an action of a cocommutative Hopf algebra.
\end{enumerate}
\end{proposition}

\pf (i) Consider the subalgebra $ZD^H$ of $D$. Set $e=[D:D^H]_l$. By Lemma \ref{dens}, $e \leq d$. On the other hand, we have:
$$\begin{array}{l}
e   = [D:D^H]_l=[D:ZD^H]_l \, [ZD^H:D^H]_l, \vspace{3pt} \\
m^2 = [D:Z]=[D:ZD^H]_l \, [ZD^H:Z].
\end{array}$$
Then $[D:ZD^H]_l$ is a common divisor of $e$ and $m^2$. Since ${\rm gcd}(d!,m)=1$, it must be $[D:ZD^H]_l=1.$ Hence $D=ZD^H$. From this, it follows that the centralizer $C_D(D^H)$ of $D^H$ in $D$ equals $Z$. Now note that $C_D(D^H)$ is $H$-stable, since one can easily check that $(h \cdot z)a = a(h \cdot z)$ for $h \in H$, $z\in C_D(D^H)$, and $a\in D^H.$ \par \medskip

(ii) It suffices to show that if the action of $H$ on $D$ is inner faithful, then $H$ is cocommutative. By (i), $D=ZD^H$ and $Z$ is $H$-stable. Let $I$ be a Hopf ideal of $H$ such that $I\cdot Z=0$. For any $h\in I, ~z\in Z, ~a\in D^H$, we have $h\cdot (za)=(h\cdot z)a=0$. Hence, $I\cdot (ZD^H)= I \cdot D=0$. Thus,
$I=0$. This shows that $H$ acts inner faithfully on a field. Applying \cite[Theorems 4.1 and 5.1]{EW1}, we obtain that $H$ is cocommutative.
\epf

\begin{remark}
(1) Proposition \ref{noaction}(ii) is a strengthening of \cite[Theorem 5.1]{EW1}, which says that the conclusion holds if $m=1.$ \par \smallskip

(2) Proposition \ref{noaction}(ii) fails in the nonsemisimple case, as there are many inner faithful actions of noncocommutative finite dimensional Hopf algebras on commutative domains; see \cite{EW2}. We conjecture (see \cite[Conjecture~5.3]{EW1}) that the result will still hold in the case that $H$ is cosemisimple, but not necessarily semisimple. \par \smallskip
\end{remark}

Notice also that when $[D:D^H]$ divides $d$, we could replace $d!$ with $d$ in Proposition \ref{noaction}. If the extension $D/D^H$ is Hopf-Galois, then $[D:D^H]=d$ and we can assume ${\rm gcd}(d,m)=1.$ Indeed, we ask:

\begin{question}
If a semisimple cosemisimple Hopf algebra $H$ over an algebraically closed field $F$ acts inner faithfully on a division algebra $D$ over $F$, then is $D/D^H$ Hopf-Galois?
\end{question}

The converse is always true for any $H$-module algebra $B$. To see this, consider the coaction $\rho: B \to B \otimes H^*$. Note that, by definition, the Galois map $can: B \otimes_{B^H} B \to B \otimes H^*$ given by $b \otimes b' \mapsto (b \otimes 1) \rho(b')$ is surjective. Hence, $\Ima \rho$ cannot land in $B \otimes (H/K)^*$ for a nonzero subspace $K$ of $H$.

\section{Semisimple Hopf actions on Weyl algebras} \label{sec:Weyl}

Recall that $k$ denotes an algebraically closed field of characteristic zero.
We are finally in a position to prove our main result:

\begin{theorem}\label{main1} Let $A:=\Weyl(k)$ be a Weyl algebra over $k$. Then, any semisimple Hopf action on $A$ factors through a group action.
\end{theorem}

\pf
We may assume that $H$ acts on $A$ inner faithfully. It is well known that the Weyl algebra $A_p$ in positive characteristic is a PI domain; namely, it is an Azumaya algebra of degree $p^n$ over its center, which is a polynomial algebra in $x_i^p$ and $y_i^p$ for $i=1,\dots,n$.
Let $D_p$ be the full localization of $A_p$. Then, $D_p$ is a division algebra of degree $p^n$.

Let $H_p$ denote the reduction of $H$ modulo $p$ from Section \ref{sec:modp}. Then $H_p$ is a semisimple
cosemisimple Hopf algebra over $\overline{\FF}_p$ by Lemma \ref{lem:Hmodp}, which acts inner faithfully on $D_p$ for sufficiently large $p$ due to Proposition \ref{prop:modp} and Lemma \ref{lem:quotdiv}. Now take $p>d$. Then Proposition \ref{noaction}(ii) implies that $H_p$ is cocommutative. Note that this is true for any choice of the homomorphism $\psi$ from Lemma \ref{lemm2}. But by Lemma \ref{lemm2}(ii), the direct product of the possible homomorphisms $\psi$ is an {\it injection} of $R$ into a direct product of fields. This implies that $H_R$ is cocommutative. Thus, $H$ is cocommutative, and hence
$H$ is a group algebra, as desired.
\epf

We extend the result above to the case when $H$ is finite dimensional, not necessarily semisimple, in the Hopf-Galois setting.

\begin{theorem}\label{main1a}
Let $A$ be a Weyl algebra over $k$, and let $H$ be a finite dimensional Hopf algebra over $k$
which acts on $A$. Assume that this action gives rise to an $H^*$-Hopf-Galois extension $A^H\subset A$. Then, the action of $H$ on $A$ factors through a group action.
\end{theorem}

\pf
We keep the notation from the previous proof. Assume that $p$ is sufficiently large.
Then by Proposition \ref{noaction}(i), the center $Z_p$ of $D_p$ is $H_p$-stable.
So, the map $\beta_p: D_p\otimes D_p\to D_p\otimes H^*_{p}$ given by $\beta_p(a\otimes b)=(a\otimes 1)\rho(b)$
restricts to an algebra map $\bar\beta_p: Z_p\otimes Z_p\to Z_p\otimes H^*_p$. Note that $\Ima\bar\beta_p$ is a $Z_p$-vector space under multiplication in the first tensor factor. Let $v_1,\dots,v_r$ be a basis of this space ($r\le \dim H_{p}$).

Let $z,z'\in Z_p$ and $c,c'\in D_p^{H_{p}}$. Then
$$\beta_p(cz\otimes c'z')=(cc'\otimes 1)\bar\beta_p(z\otimes z').$$
By Proposition \ref{noaction}(i), one has $Z_pD_p^{H_{p}}=D_p$. Hence, $\Ima\beta_p$ is spanned by $v_1,\dots,v_r$
as a left $D_p$-vector space (under multiplication in the first tensor factor).

Now, since the action of $H$ on $A$ gives rise to a Hopf-Galois extension,  the action of $H_p$ on $A_p$, and hence the action of $H_p$ on $D_p$, gives rise to a Hopf-Galois extension as well,
i.e., $\Ima\beta_p=D_p\otimes H^*_{p}$. This yields $r\ge \dim H_{p}$. Thus, $r=\dim H_{p}$ and $\Ima\bar\beta_p=Z_p\otimes H_p^*$. Therefore, $H_p^*$ is commutative and $H_p$ is cocommutative. So we conclude as in the proof of Theorem \ref{main1}
that $H$ is cocommutative, hence a group algebra.
\epf

\begin{proposition}\label{main1b}
Theorems \ref{main1} and \ref{main1a} remain true if $A$ were replaced by ${\bold A}_n(k[z_1,\dots,z_s])$, a Weyl algebra over a polynomial algebra.
\end{proposition}

\pf
The proof is analogous to that of Theorem \ref{main1} and Theorem \ref{main1a}.
\epf

\begin{proposition}\label{main1c}
Proposition \ref{main1b} remains true if the module algebra \linebreak $A:={\bold A}_n(k[z_1, \dots, z_s])$ were replaced by the quotient division algebra ${\mathcal Q}_A$ of $A$.
\end{proposition}

\pf
The proof is along the lines of those of Theorems \ref{main1} and \ref{main1a}. Let us describe the necessary changes. As in the proof of Lemma \ref{lem:R-order}, we can write $h_i\cdot x_j=P_{ij}, \ h_i\cdot y_j=Q_{ij}$, and $h_i\cdot z_\ell=R_{i\ell},$ where now $P_{ij},Q_{ij},R_{i\ell} \in {\mathcal Q}_A$. There is a common denominator $T \in A$ such that $P_{ij}=T^{-1}P_{ij}',$ \linebreak $Q_{ij}=T^{-1}Q_{ij}'$, and $R_{i\ell}=T^{-1}R_{i\ell}'$, with $P_{ij}',Q_{ij}',R_{i\ell}'\in A$.
For a sufficiently large prime $p$, these formulas can be reduced modulo $p$. Set now $A_p$ for the Weyl algebra over the given polynomial algebra in characteristic $p$ (reduction of $A$ modulo $p$) and $D_p$ for its quotient division algebra. We can define an algebra map $\rho: A_p \to D_p\otimes H_p^*$ such that $\rho(T)$ is invertible.

Let us show that $\rho$ extends to a coaction of $H_p^*$ on $D_p$, i.e., that $\rho(a)$ is invertible for any nonzero $a \in A_p$. To this end, let ${\mathcal Q}_{\rho,A_p}\subset D_p$ be the partial localization of $A_p$ obtained by inverting all elements $a\in A_p$ such that $\rho(a)$ is invertible.
Then, $\rho$ extends to an algebra map $\rho': {\mathcal Q}_{\rho,A_p}\to D_p\otimes H_p^*$. Moreover, $P_{ij},Q_{ij},R_{i\ell}\in {\mathcal Q}_{\rho,A_p}$, and
\begin{equation}\label{formu}
h_r \cdot P_{ij}=\sum_m c_{ri}^mP_{mj},\quad h_r \cdot Q_{ij}=\sum_m c_{ri}^mQ_{mj},\quad h_r \cdot R_{i\ell}=\sum_m c_{ri}^mR_{m\ell},
\end{equation}
where $c_{rj}^m$ are the structure constants of the multiplication of $H_p$.
Let $B$ be the subalgebra of $D_p$ generated by $P_{ij},Q_{ij},$ and $R_{i\ell}$.
Write $\{h_i^*\}_i$ for the dual basis of $\{h_i\}_i$. Since
$$
x_j=\sum_i h_i^*(1)P_{ij},\quad  y_j=\sum_i h_i^*(1)Q_{ij},\quad z_\ell=\sum_i h_i^*(1)R_{i\ell},
$$
we have $A_p\subset B$. Also, $B\subset {\mathcal Q}_{\rho,A_p}$, and by \eqref{formu}, $h_r \cdot B \subset B$ for all $r$. Thus, the map $\rho'$ defines a coaction $B\to B\otimes H_p^*$. By \cite[Theorem 2.2]{SV} (which applies because $B$ is contained in a division algebra), $\rho'$ extends to a coaction $\rho'':{\mathcal Q}_B\to {\mathcal Q}_B\otimes H_p^*$.
But ${\mathcal Q}_B=D_p$, as $A_p\subset B\subset D_p$, so we get a coaction $D_p\to D_p\otimes H_p^*$. (In particular, we see that ${\mathcal Q}_{\rho,A_p}=D_p$.)

Now we proceed as in the proofs of Theorems \ref{main1} and \ref{main1a}.
In particular, in Theorem \ref{main1a}, to establish the surjectivity of $\beta_p: D_p\otimes D_p\to D_p\otimes H_p^*$ for large $p$ we modify the argument as follows: pick $a_i\in {\mathcal Q}_A\otimes {\mathcal Q}_A$ such that $\beta(a_i)=1\otimes h_i^*$ (where $\beta: {\mathcal Q}_A\otimes {\mathcal Q}_A\to {\mathcal Q}_A\otimes H^*$ is the map in characteristic zero). There is a finite number of fractions involved when expressing $a_i$ as a sum of elements of the form $b_j \otimes b'_j$ with $b_j,b'_j \in {\mathcal Q}_A.$ Arguing as before with a common denominator, we can choose $p$ so large that the formula $\beta(a_i)=1\otimes h_i^*$ can be reduced modulo $p$. Then $1\otimes h_i^*\in \Ima\beta_p$ and $\beta_p$ is surjective.
\epf

\begin{remark} Note that in the proof of Proposition~\ref{main1c}, we do not reduce ${\mathcal Q}_A$ modulo $p$ (which is not a well-behaved construction),
but rather reduce modulo $p$ the explicit formulas defining the action of $H$ on the generators of $A$.
\end{remark}


\section*{Acknowledgments}
\noindent J. Cuadra was supported by grant MTM2011-27090 from MICINN and FEDER and by the research group FQM0211 from Junta de Andaluc\'{\i}a. P. Etingof and C. Walton were supported by the US National Science Foundation, grants DMS-1000113 and DMS-1401207, respectively.

This work was mainly done during the visit of the first named author to the Mathematics department of MIT. He is deeply grateful for the warm and generous hospitality and excellent working conditions.

Finally, we are very grateful to the referee for useful suggestions.


\end{document}